\documentclass[11pt,leqno]{article}
\usepackage{amssymb, amscd, amsmath, amsthm}

\allowdisplaybreaks

\def\ve{\varepsilon}
\def\intl{\int\limits}
\def\mod{\,\text{\rm mod}\;}
\def\beq{\begin{equation}}
\def\eeq{\end{equation}}
\def\cite#1{{\rm [#1]}}

\def\wt{\widetilde }

\newtheorem{theorem}{Theorem}
\newtheorem{lemma}{Lemma}

\theoremstyle{definition}

\newtheorem*{defin}{Definition}
\newtheorem*{rema}{Remark}

\begin{document}

\numberwithin{equation}{section}
\title{On small gaps between primes and almost prime powers}

\author{J\'anos Pintz\thanks{Supported by OTKA Grants K72731, K67676 and ERC-AdG.228005.}
}

\date{}
\maketitle
\setcounter{section}{1}

1. In two subsequent works, joint with D. Goldston and C. Y. Y{\i}ld{\i}r{\i}m \cite{GPY1, GPY2} we showed that for the sequence $p_n$ of primes
\beq
\liminf_{n \to \infty} \frac{p_{n + 1} - p_n}{\log p_n} = 0,
\label{eq:1.1}
\eeq
and even
\beq
\liminf_{n \to \infty} \frac{p_{n + 1} - p_n}{(\log p_n)^{1/2}(\log\log p_n)^2} < \infty.
\label{eq:1.2}
\eeq
A crucial ingredient of the proof was the celebrated Bombieri--Vinogradov theorem, which asserts that $\vartheta = 1/2$ is an admissible level of distribution of primes, that is,
\beq
\sum_{q \leq N^\vartheta /\log^C\!\! N} \max_{\substack{a\\ (a,q) = 1}} \biggl|\sum_{\substack{p\equiv a\,(\!\mod q)\\ p\leq N}} 1 - \frac{li\, N}{\varphi(q)}\biggr| \ll_{A} \frac{N}{\log^A\!\! N}
\label{eq:1.3}
\eeq
holds with $\vartheta = 1/2$ for any $A > 0$, $C > C(A)$.
The method also yielded \cite{GPY1} that if $\vartheta > 1/2$ is an admissible level of distribution of primes then for any \emph{admissible} $k$-element set $\mathcal H = \{h_i\}^k_{i = 1}$ (that is, if $\mathcal H$ does not occupy all residue classes $\mod p$ for any prime $p$) the set $n + \mathcal H := \{n + h_i \}^k_{i = 1}$ contains at least two primes for infinitely many values of $n$ if $k \geq k_0(\vartheta)$.
Consequently we have infinitely many bounded gaps between primes, more precisely
\beq
\liminf_{n \to \infty} (p_{n + 1} - p_n) \leq C(\vartheta).
\label{eq:1.4}
\eeq
The strongest possible hypothesis on the uniform distribution of primes in arithmetic progressions, the Elliott--Halberstam \cite{EH} conjecture stating the admissibility of the level $\vartheta = 1$ (with $N/\log^C\!\! N$ replaced by $N^{1 - \ve}$ for any $\ve > 0$), or slightly weaker, even the assumption $\vartheta \geq 0.971$ implies gaps of size at most 16 infinitely often, in fact,
\beq
k_0(0.971) = 6, \qquad C(0.971) = 16.
\label{eq:1.5}
\eeq
If $\vartheta = 1/2 + \delta$ is near to $1/2$, that is, $\delta$ is a small positive number, one can take for $\delta \to 0^+$
\beq
k_0 \left(\frac12 + \delta\right) = \left( 2 \left\lceil \frac1{2\delta}\right\rceil + 1 \right)^2, \qquad C \left(\frac12 + \delta\right) \sim 2\delta^{-2} \log \frac1{\delta}.
\label{eq:1.6}
\eeq

This situation suggests that one might take some prime-like set $\mathcal P'$ just slightly more dense than the set $\mathcal P$ of primes, that is, for any $\ve > 0$ a set $\mathcal P_{\ve}$ such that
\beq
\mathcal P \subset \mathcal P'_\ve := \{b_n\}^\infty_{n = 1},\quad
\pi'_\ve(N) := \#\{n \leq N, \ n \in \mathcal P'\} < \pi(n)(1 + \ve)
\label{eq:1.7}
\eeq
which has bounded gaps infinitely often, that is,
\beq
\liminf_{n \to \infty} (b_{n + 1} - b_n) < \infty.
\label{eq:1.8}
\eeq

Of course adding $p + 1$ to the set $\mathcal P$ for infinitely many primes would trivially satisfy the requirements but we are looking for some arithmetically interesting set $\mathcal P'_\ve$ with some similarity to primes or prime powers.
(Adding just prime powers to $\mathcal P$ raises the number of elements just with a quantity $\sim 2N^{1/2} / \log N$ which is negligible compared to $\pi(N)$.)
One possibility is to add some numbers which are similar to prime powers.
To avoid confusion with almost primes we will introduce the following

\begin{defin}
For any $\ve \geq 0$ a natural number $n$ is called $\ve$-balanced if for any prime divisors $p, q$ of $n$ we have
\beq
\min(p, q) \geq \bigl(\max(p,q)\bigr)^{1 - \ve}.
\label{eq:1.9}
\eeq
\end{defin}

\begin{rema}
With this definition $0$-balanced numbers larger than~$1$ are exactly the primes and prime powers.
\end{rema}

Let us denote the set of $\ve$-balanced numbers by $\mathcal P_\ve$, the total number of prime divisors of~$n$ by $\Omega(n)$ and let
\beq
\mathcal P_{\ve, r} := \bigl\{n \in \mathcal P_\ve, \ \Omega(n) = r\bigr\}, \ \ \ \mathcal P_\ve := \bigcup^\infty_{r = 1} \mathcal P_{\ve, r}.
\label{eq:1.10}
\eeq
(In this way we can talk about almost prime-squares $(r = 2)$, almost prime-cubes $(r = 3)$ etc.)

To have an idea about the quantity
\beq
\pi_{\ve, r} (N) := \#\{ N \leq n < 2N;\ n \in \pi_{\ve, r} (N) \}
\label{eq:1.11}
\eeq
we remark that denoting by $P^-(n)$ and $P^+(n)$ the least, resp., the greatest prime factor of~$n$ we have obviously
\beq
n \in \pi_{\ve, r} (N) \Longrightarrow N^{(1 - \ve)/r} \leq  P^-(n) \leq P^+(n) \leq (2N)^{1/(r(1 - \ve))}.
\label{eq:1.12}
\eeq
Reversed, we have also clearly for $n \in [N, 2N)$,  $\Omega(n) = r$ by $(1 + \ve/2)(1 - \ve) \leq 1 - \ve/2$
\beq
N^{(1 - \ve/2)/r} \leq  P^-(n) \leq P^+(n) \leq N^{(1 + \ve/2)/r} \Longrightarrow n \in \pi_{\ve, r}(N).
\label{eq:1.13}
\eeq

In order to simplify the calculation of the density of the $\ve$-balanced numbers we will work with the smaller subsets of $\mathcal P_{\ve, r}$, defined by
\begin{align}
\label{eq:1.14}
\mathcal P^*_{r, \ve}(N) :={} & \Bigl\{N \leq n < 2N, \ \Omega(n) = r, \\
&\qquad N^{(1 - \ve/2)/r} \leq P^-(n) \leq P^+(n) \leq N^{(1 + \ve/2)/r}\Bigr\}.\notag
\end{align}

The prime number theorem implies with easy calculations that by
\begin{gather*}
a_1 := (1 - \ve/2)/r,\qquad a_2 := (1 + \ve/2)/r,\\
I := \bigl[N^{a_1}, N^{a_2}\bigr],\qquad
J(\bold u) := (N/u_1 \dots u_{r - 1}, 2N/u_1 \dots u_{r - 1}]
\end{gather*}
\begin{align}
\label{eq:1.15}
\pi^*_{r,\ve}(N)
:&= \#\bigl\{n \in \mathcal P^*_{r,\ve}(N) \bigr\} = \sum_{\substack{N \leq p_1 \dots p_r < 2N\\
p_i \in I}} 1 \sim \\
&\sim \intl_I \dots \intl_I \prod^{r = 1}_{i = 1} \frac1{\log u_i} \intl_{I \cap J(\bold u)} \frac1{\log t} du_1 \dots du_{r - 1} \, dt \sim \notag\\
&\sim \frac{N}{\log N} {\intl^{a_2}_{a_1} \dots
\intl^{a_2}_{a_1}} \frac{d\alpha_1 \dots d \alpha_{r - 1}}{\alpha_1 \dots \alpha_{r - 1} (1 - \alpha_1 - \dots - \alpha_{r - 1})} =: \frac{C_0(r, \ve)N}{\log N}. \notag
\end{align}
Here we have obviously for $\ve \to 0$
\beq
C_0(r, \ve) \leq \left(\frac{\ve}{r}\right)^{r - 1} \frac{r^r}{(1 - \ve/2)^r} = \frac{r\ve^{r - 1}}{(1 - \ve/2)^r}.
\label{eq:1.16}
\eeq

Since for $\ve < \ve_0$ we have $\mathcal P_{r,\ve}(N) \subset \mathcal P^*_{r,3\ve}(N)$ the above assertion shows that the number of $\ve$-balanced composite numbers (the counting function of $\mathcal P'_\ve \setminus \mathcal P$) is negligible compared to that of the primes, since even in total
\beq
\sum^\infty_{r = 2} C_0(r,\ve) < 3\ve \ \text{ if } \ \ve < c_0.
\label{eq:1.17}
\eeq

After this preparation we can formulate our result.

\begin{theorem}
\label{th:1}
Let $r = 2$ or $3$, $\ve > 0$.
Then the set of $\ve$-balanced numbers with either one or $r$ prime factors contains infinitely many bounded gaps, but has $(1 + O(\ve)) \pi(N)$ elements below~$N$.
\end{theorem}

\bigskip
\setcounter{section}{2}
\setcounter{equation}{0}
2. We will actually prove a stronger result.

\begin{theorem}
\label{th:2}
Let $r = 2$ or $3$, $\ve > 0$ and let $\mathcal H$ be an arbitrary $k$-element admissible set of non-negative integers, $k > k_0(\ve)$.
Then the $k$-tuple $n + \mathcal H$ contains at least two $\ve$-balanced numbers with either one or $r$ prime factors for infinitely many values of~$n$. \end{theorem}

\begin{proof}
Similarly to the role of the Bombieri--Vinogradov theorem \eqref{eq:1.3} in the proof of \eqref{eq:1.1}--\eqref{eq:1.2} we need the analogous assertion for the $\ve$-balanced numbers in $\mathcal P^*_{r,\ve}(N)$ defined in \eqref{eq:1.14}.
\end{proof}

\begin{theorem}
\label{th:3}
We have for any $A > 0$ with $C > C(A)$
\beq
\label{eq:2.1}
\sum_{q\leq \sqrt{N}/\log^C \!\! N} \max_{\substack{a\\
(a,q) = 1}} \biggl| \sum_{\substack{n \equiv a(q)\\
n \in \mathcal P^*_{r,\ve}(N)}} 1 - \frac{C_0(r,\ve) li\, N}{\varphi(q)} \biggr| \ll_{A, r} \frac{N}{\log^A \!\! N}.
\eeq
\end{theorem}

The proof runs analogously to the proof of Vaughan \cite{Vau} of the Bombieri--Vinogradov theorem or one may apply some form of generalized Bombieri--Vinogradov type theorems, as that of Y. Motohashi \cite{Mot} or Pan Cheng Dong \cite{Pan}.
The latter asserts that for any $\alpha > 0$, $\ve > 0$ and any $f(m) \ll 1$ we have
\beq
\sum_{q \leq \sqrt{N}/\log^C\!\! N} \max_{\substack{a\\
(a,q) = 1}} \Biggl| \sum_{m\leq N^{1 - \alpha}} f(m) \biggl(\sum_{\substack{mp \leq N\\
mp \equiv a(\mod q)}} 1 - \frac{li\, \frac{N}{m}}{\varphi(q)} \biggr) \Biggr| \ll_{\alpha, A} \frac{N}{\log^A \!\! N}.
\label{eq:2.2}
\eeq

The work \cite{GPY1} was based on two main lemmas describing properties of the crucial weight function $\Bigl(\mathcal H = \{h_i\}^k_{i = 1}\Bigr)$
\beq
\Lambda_R(n; \mathcal H, l) = \frac1{(k + l)!} \sum_{d\mid P_{\mathcal H}(n), d \leq R} \mu(d) \log^{k + l} \frac{R}{d}, \ \ P_{\mathcal H}(n) := \prod^k_{i = 1} (n + h_i).
\label{eq:2.3}
\eeq
The formulation of the main lemmas need the singular series
\beq
\mathfrak S(\mathcal H) = \prod \left(1 - \frac{\nu_p(\mathcal H)}{p} \right) \left(1 - \frac1p\right)^{-k},
\label{eq:2.4}
\eeq
where $\nu_p(\mathcal H)$ denotes the number of residue classes occupied by $\mathcal H \mod p$, for any prime~$p$.
The admissible property of $\mathcal H$ means $\nu_p(\mathcal H) < p$ for any $p$, or equivalently $\mathfrak S(\mathcal H) \neq 0$.
The two main lemmas below are special cases of Propositions~1 and 2 of \cite{GPY1}.

In the following let $\eta > 0$, $k$, $l$ bounded, but arbitrarily large integers, $n \sim N$ substitutes $n \in [N, 2N)$
\beq
\max_{h_i \in \mathcal H} h_i \ll \log N, \quad R > N^{c_0}, \quad \chi_{\mathcal P}(n) = \begin{cases} 1 &\text{if }\ n \in \mathcal P,\\
0 &\text{if }\ n \notin \mathcal P.
\end{cases}
\label{eq:2.5}
\eeq

\begin{lemma}
\label{lem:1}
For $R \leq \sqrt{N}/(\log N)^C$, $N \to \infty$, we have
\beq
\sum_{n \sim N} \Lambda_R(n; \mathcal H, k + l)^2 = {2l\choose l} \frac{N(\log R)^{k + 2l} (\mathfrak S(\mathcal H) + o(1))}{(k + 2l)!}.
\label{eq:2.6}
\eeq
\end{lemma}

\begin{lemma}
\label{lem:2}
For $h \in \mathcal H$, $R \leq N^{1/4} /(\log N)^C$, $C > C(A)$, $N \to \infty$, we have
\beq
\sum_{n \sim N} \Lambda_R(n; \mathcal H, k + l)^2 \chi_{\mathcal P}(n + h) = {2l + 2\choose l + 1} \frac{N(\log R)^{k + 2l + 1} (\mathfrak S(\mathcal H) + o(1))}{(k + 2l + 1)!\log N}.
\label{eq:2.7}
\eeq
\end{lemma}

In the proof of Lemma~\ref{lem:2} actually just two properties of the primes are used:

(i) their distribution in residue classes is on average regular as described by the Bombieri--Vinogradov theorem;

(ii) if $n + h_0 \in \mathcal P$, $n \sim N$, then $\mathcal P_{\mathcal H}(n)$ and $\mathcal P_{\mathcal H \setminus \{ h\}}(n)$ have the same divisors below $R$, that is, $n + h_0$ has no prime divisor below~$R$.

\medskip
The first property is shared by the elements of $\mathcal P^*_{r,\ve}(N)$ as shown by \eqref{eq:2.1}, the only change being the factor $C_0(r,\ve)$.
In the cases $r = 2$ and $r = 3$ they obviously share property (ii) as well.

In such a way with the notation
\beq
\mathcal P(N) = [N, 2N) \cap \mathcal P, \quad \wt{\mathcal P}_{r,\ve}(N) = \mathcal P(N) \cup \mathcal P^*_{r,\ve}(N)
\label{eq:2.8}
\eeq
we obtain in exactly the same way as Lemma~\ref{lem:2}, for the characteristic function $\chi_{\wt{\mathcal P}}$ of the set $\wt{\mathcal P}$ the following

\begin{lemma}
\label{lem:3}
For $R \leq N^{1/4}/(\log N)^C$, $C > C(A, r, \ve)$, $r = 2$ or $3$, $N \to \infty$, we have
\begin{align}
\label{eq:2.9}
&\sum_{n \sim N} \Lambda_R(n; \mathcal H, l)^2 \chi_{\wt{\mathcal P}} (n + h_0) =\\
&= {2l + 2\choose l + 1} \frac{N(\log R)^{k + 2l + 1} \mathfrak S(\mathcal H) (1 + C_0(r, \ve) + o(1))}{(k + 2l + 1)!\log N}.
\notag
\end{align}
\end{lemma}

In this case we have, similarly  to (3.3) of \cite{GPY1},
\begin{align}
\label{eq:2.10}
S :&= \sum_{n \sim N} \biggl(\sum^k_{i = 1} \chi_{\wt{\mathcal P}}(n + h_i) - 1\biggr) \Lambda_R(n; \mathcal H, l)^2 \sim\\
&\sim {2l \choose l} \frac{N(\log R)^{k + 2l} \mathfrak S(\mathcal H)}{(k + 2l)!} \left(\frac{k}{k + 2l + 1} \cdot \frac{2l + 1}{2l + 2} (1 + C_0(r, \ve)) - 1 \right) > 0 \notag
\end{align}
if we choose $l = \left\lfloor \sqrt{k}/2\right\rfloor$, $k > k_0(r, \ve)$, which proves Theorem~\ref{th:2}, consequently also Theorem~\ref{th:1} for $r = 2,3$.

\bigskip

\noindent
{\small J\'anos {\sc Pintz}\\
R\'enyi Mathematical Institute of the Hungarian Academy
of Sciences\\
Budapest\\
Re\'altanoda u. 13--15\\
H-1053 Hungary\\
E-mail: pintz@renyi.hu}

\end{document}